\documentclass[12 pt, letterpaper]{amsart}

\usepackage{hyperref}
\usepackage{etex}
\usepackage[shortlabels]{enumitem}
\usepackage{amsmath}
\usepackage{amsxtra}
\usepackage{amscd}
\usepackage{amsthm}
\usepackage{amsfonts}
\usepackage{amssymb}
\usepackage{eucal}
\usepackage[all]{xy}
\usepackage{graphicx}
\usepackage{tikz-cd}
\usepackage{mathrsfs}
\usepackage{subfiles}
\usepackage{mathpazo}
\usepackage[colorinlistoftodos, textsize=tiny]{todonotes}
\usepackage{morefloats}
\usepackage{pdfpages}
\usepackage{thm-restate}
\usepackage[utf8]{inputenc}
\usepackage{epigraph}
\usepackage{csquotes}
\usepackage[margin= 1.0in]{geometry}
\usepackage{adjustbox}
\usepackage{scalerel}
\usepackage{stackengine}
\stackMath
\newcommand\reallywidehat[1]{%
\savestack{\tmpbox}{\stretchto{%
  \scaleto{%
    \scalerel*[\widthof{\ensuremath{#1}}]{\kern-.6pt\bigwedge\kern-.6pt}%
    {\rule[-\textheight/2]{1ex}{\textheight}}
  }{\textheight}%
}{0.5ex}}%
\stackon[1pt]{#1}{\tmpbox}%
}
\graphicspath{ {images/} }

\RequirePackage{color}
\definecolor{myred}{rgb}{0.75,0,0}
\definecolor{mygreen}{rgb}{0,0.5,0}
\definecolor{myblue}{rgb}{0,0,0.65}

\usepackage{color}

\usepackage{hyperref}
\hypersetup{citecolor=blue}
\usepackage{tikz}
\usetikzlibrary{matrix,arrows,decorations.pathmorphing}

\theoremstyle{plain}
\newtheorem{theorem}[subsection]{Theorem}

\newtheorem{proposition}[subsection]{Proposition}
\newtheorem{lemma}[subsection]{Lemma}

\theoremstyle{definition}

\newtheorem{remark}[subsection]{Remark}

\newtheorem{question}[subsection]{Question}

\theoremstyle{remark}
\newtheorem{notation}[subsection]{Notation}

\numberwithin{equation}{section}
\newcommand\nc{\newcommand}
\nc\on{\operatorname}
\nc\renc{\renewcommand}

\newcommand*{\shom}{\mathscr{H}\kern -.5pt om}
\newcommand*{\stor}{\mathscr{T}\kern -.5pt or}
\newcommand*{\sext}{\mathscr{E}\kern -.5pt xt}

\makeatletter
\providecommand\@dotsep{5}
\renewcommand{\listoftodos}[1][\@todonotes@todolistname]{%
\@starttoc{tdo}{#1}}
\makeatother

\makeatletter
\newcommand{\customlabel}[2]{\protected@write \@auxout {}{\string \newlabel {#1}{{#2}{\thepage}{#2}{#1}{}} }\hypertarget{#1}{#2}}

\renewcommand\hom{\mathrm{Hom}}

\DeclareFontFamily{U}{wncy}{}
\DeclareFontShape{U}{wncy}{m}{n}{<->wncyr10}{}
\DeclareSymbolFont{mcy}{U}{wncy}{m}{n}
\DeclareMathSymbol{\Sha}{\mathord}{mcy}{"58}

\makeatletter
\@namedef{subjclassname@2020}{%
  \textup{2020} Mathematics Subject Classification}
\makeatother

\def\listtodoname{List of Todos}
\def\listoftodos{\@starttoc{tdo}\listtodoname}

\title{Prill's problem}
\author{Aaron Landesman and Daniel Litt}
\AtEndDocument{\bigskip{\footnotesize%
  \textsc{Department of Mathematics, MIT, 
182 Memorial Dr,   
\mbox{Cambridge, MA 02142}} \par
Aaron Landesman: \texttt{aaronl@mit.edu} \par
    \addvspace{\medskipamount}
  \textsc{Department of Mathematics, University of Toronto, 
	  Bahen Centre,
40 St. George St.,
  \mbox{Toronto, Ontario, Canada, M5S 2E4}}  \par
  Daniel Litt: \texttt{daniel.litt@utoronto.ca} \par
  }}

\usepackage{microtype}
\begin{document}

\begin{abstract}
	We solve Prill's problem, originally posed by David Prill in the late
	1970s and popularized in ACGH's
``Geometry of Algebraic Curves." 
That is, for any curve $Y$ of genus $2$,	
we produce a finite \'etale degree $36$ connected cover $f: X \to Y$ where, for every point $y \in Y$, $f^{-1}(y)$ moves in a pencil.
\end{abstract}

\maketitle

\section{Introduction}

Throughout we work over the complex numbers. Let $f: X\to Y$ be a dominant map
of smooth projective connected algebraic curves, where the genus of $Y$, $g(Y),$ is at
least $2$. In this case, Riemann-Hurwitz yields $\deg(f)<g(X)$,
so one would not expect that $f^{-1}(y)$ moves in a pencil for a general $y\in Y$. 
That is, one expects $h^0(X, \mathscr{O}_X(f^{-1}(y)))=1$ for a general $y
\in Y$. 
Of course, there are many cases where $f^{-1}(y)$ moves in a pencil for special
$y \in Y$, but in the late 1970s, David Prill raised the following question:
\begin{question}[Prill's problem, ~\protect{\cite[p. 268, Chapter VI, Exercise
	D]{ACGH:I}}]
	\label{question:}
	Given any curve $Y$ of genus $g \geq 2$ and a finite covering
	$f: X \to Y$, does 
	$h^0(X, \mathscr{O}_X(f^{-1}(y)))=1$ for a general $y \in Y$?
\end{question}

Due to its elementary nature, Prill's problem garnered much attention
in the early 1980s. 
Various special cases of Prill's problem were answered affirmatively, such as those summarized in
\cite{ACGH:I, biswasB:on-prills-problem,kumar:some-remarks-on-prills-problem}.

We say a finite cover $f: X \to Y$ of smooth proper geometrically connected
curves is {\em Prill exceptional} if 
$h^0(X, \mathscr O_X(f^{-1}(y))) \geq 2$ for every point $y \in Y$.
The general belief, up until this point, was that 
no Prill exceptional covers
should exist. 
Our main result unexpectedly resolves Prill's problem 
by showing that Prill exceptional covers do, in fact, exist. 
Even more surprisingly,
our construction
gives a Prill exceptional cover
of {\em any} genus $2$ curve.

\begin{theorem}
	\label{theorem:prill-example}
	If $Y$ is any smooth proper connected curve of genus $2$ over the
	complex numbers, there is a
	finite \'etale cover $f: X \to Y$ which is Prill exceptional.
\end{theorem}

The idea of the proof of \autoref{theorem:prill-example} is to relate Prill's problem to a problem in Hodge theory. Inspired by
recent work of Markovi\'c \cite{markovic}, we employ a construction of Bogomolov and
Tschinkel \cite{BT:curve-correspondences, BT:updated-curve-correspondences}.
This yields a finite \'etale
cover $f: X\to Y$ of degree $36$ such that the Jacobian $\on{Pic}^0_{X}$ has an isogeny factor which is independent of the complex structure on $Y$. Analyzing the infinitesimal variation of Hodge structure associated to $H^1(X, \mathbb{Q})$ yields that $f_*{\omega}_X$ is not generically globally generated, and a routine calculation shows that $f$ is Prill exceptional.

\begin{remark}
If $\psi : X'\to Y$ is a finite cover such that $\psi$ factors through
a Prill exceptional cover $f: X\to Y$, then $\psi$ is also Prill exceptional, as
there is an injection $H^0(X, \mathscr{O}(f^{-1}(y)))\to H^0(X', \mathscr{O}(g^{-1}(y)))$. Thus \autoref{theorem:prill-example} can be used to construct Prill exceptional covers of arbitrary degree, by composing with an arbitrary map $X'\to X$.
\end{remark}

\begin{remark}
	\label{remark:}
	Although \autoref{theorem:prill-example} solves Prill's problem, it
would still be extremely interesting to know if there any Prill exceptional
covers where $Y$ has genus $g > 2$.
\end{remark}

\begin{remark}
	\label{remark:}
In this paper, we give a short solution to Prill's problem, a simple
to state question, which has been open since the 1970s.
In a companion paper, \cite[Remark 5.7]{landesmanL:applications-putman-wieland},
we give a much more involved proof
of a slightly weaker result, 
yielding a Prill exceptional cover of a general genus 2 curve. 
That proof relies on heavier machinery, but
we decided to also include it in
\cite{landesmanL:applications-putman-wieland}
as it is nearly automatic from the
tools developed there. 
\end{remark}

\subsection{Acknowledgements}
We especially thank Anand Deopurkar for alerting us to Prill's problem, and Rob Lazarsfeld for encouraging us to write this short note and for giving us useful comments on it.
We would like to thank Ravindra Girivaru, Joe Harris, 
Neithalath Mohan Kumar and Rob Lazarsfeld
for telling us about their previous work on Prill's problem.
We thank Mihnea Popa for a helpful discussion.
Landesman was supported by the National Science Foundation under Award No.
DMS-2102955. Litt was supported by an NSERC Discovery Grant, ``Anabelian methods in arithmetic and algebraic geometry."

\section{The Proof of \autoref{theorem:prill-example}}

To start, we rephrase the condition that $f: X \to Y$ is Prill exceptional in terms of
$f_* \omega_X$ not being generically globally generated, meaning that all of its
global sections lie in a proper subbundle.
This lemma is known to experts, but we recall it for completeness.
\begin{lemma}
	\label{proposition:ggg-and-prill}
	Let $f: X\to Y$ be a finite \'etale morphism of smooth proper connected curves.
	Then $f_*\omega_X$ is not generically globally generated if and only if $f$ is Prill exceptional.
\end{lemma}
\begin{proof}
We wish to show that $f_* \omega_X$ is not generically globally generated
if and only if $h^0(X, \mathscr{O}_X(f^{-1}(y)))=h^0(Y, (f_* \mathscr O_X)(y)) > 1$ for every point $y$ in $Y$.
Note that $f_* \omega_X$
is not generically globally generated if and only if, for a general $y \in Y$, we have an exact sequence
\begin{equation}
	\label{equation:}
	\begin{tikzcd}[column sep=small, row sep=1 pt]
		0 \ar {r} & H^0(Y, f_* \omega_X(-y)) \ar {r} & H^0(Y, f_*
		\omega_X) \ar {r}{\beta}  & H^0(Y, f_* \omega_X|_y)
		\ar{r}& \qquad\\
		\qquad \ar{r} & H^1(Y, f_* \omega_X(-y)) \ar {r} & H^1(Y, f_* \omega_X) \ar {r} &
		0,
\end{tikzcd}\end{equation}
where $\beta$ is not surjective.
Since  $h^1(Y, f_* \omega_X) = h^1(X, \omega_X)= 1$, the map $\beta$ is not surjective precisely when
$h^1(Y, f_* \omega_X(-y)) > 1$.
By Serre duality, this is equivalent to
$h^0(Y, f_* \mathscr O_X(y)) > 1$ for a general point $y \in Y$. 
This is equivalent to the statement that
$h^0(Y, f_* \mathscr O_X(y)) > 1$ for all points $y \in Y$
by upper semicontinuity of sheaf cohomology.
\end{proof}

Throughout the remainder of the proof, we work in the following setup.
\begin{notation}
	\label{notation:versal}
Consider a diagram of the form
\begin{equation}
	\label{equation:versal}
	\begin{tikzcd}
		\mathscr X \ar {rr}{h} \ar[rd, "\pi'"'] && \mathscr Y \ar
		{ld}{\pi} \\
		& \mathscr M & 
	\end{tikzcd}
\end{equation}
where $\pi$ is a relative smooth proper curve of
genus $g\geq 2$ with geometrically connected fibers, $h$ is finite \'etale, and
$\pi'$ is a smooth proper curve of genus $g'$ with geometrically connected
fibers. Suppose that the map $\mathscr{M}\to\mathscr{M}_g$ induced by $\pi$ is
dominant \'etale. Fix $m\in \mathscr{M}$ and let $X=\mathscr{X}_m,
Y=\mathscr{Y}_m,$ and $f=h|_X$. We refer to the data of
\eqref{equation:versal} as a \emph{versal family of covers of curves of genus $g$}.
\end{notation}
The next proposition shows that, in the above setup, in order to construct Prill
exceptional curves, it is enough to
produce an isotrivial isogeny factor in $\on{Pic}^0_{\mathscr
X/\mathscr M}$.
\begin{proposition}
	\label{proposition:non-ggg}
	With notation as in \autoref{notation:versal}, suppose that the Jacobian
	$\on{Pic}^0_{\mathscr{X}/\mathscr{M}}$ has an isotrivial isogeny factor. Then, $f_*\omega_X$ is not generically globally generated. 
\end{proposition}
\begin{proof}
	Let $\mathbb{V}=R^1\pi'_*\mathbb{Q}$ and $\mathscr{V}=\mathbb{V}\otimes \mathscr{O}_\mathscr{M}$. 
	We study the infinitesimal variation of Hodge structure associated to
	$\mathbb{V}$. The Hodge filtration on $\mathscr{V}$ satisfies 
	\begin{align*}
		F^1\mathscr{V}&=\pi'_*\omega_{\mathscr{X}/\mathscr{M}}
		\hspace{.5cm} \text{and} \hspace{.5cm}
\mathscr{V}/F^1\mathscr{V}=R^1\pi'_*\mathscr{O}_{\mathscr{X}}.
	\end{align*}
	The Gauss-Manin connection $\nabla: \mathscr{V}\to \mathscr{V}\otimes \Omega^1_{\mathscr{M}}$ induces an $\mathscr{O}_\mathscr{M}$-linear map $$\overline{\nabla}: \pi'_*\omega_{\mathscr{X}/\mathscr{M}}=F^1\mathscr{V}\to \mathscr{V}/F^1\mathscr{V}\otimes \Omega^1_{\mathscr{M}}=R^1\pi'_*\mathscr{O}_{\mathscr{X}}\otimes \Omega^1_{\mathscr{M}}.$$
	
	The given isotrivial isogeny factor of
	$\on{Pic}^0_{\mathscr{X}/\mathscr{M}}$ yields a nonzero isotrivial sub-$\mathbb{Q}$-Hodge structure $\mathbb{W}\subset \mathbb{V}$. Hence $\overline{\nabla}$ has nontrivial kernel: the kernel contains $F^1(\mathbb{W}\otimes \mathscr{O}_{\mathscr{M}})$.
	
	Restricting $\overline{\nabla}$ to the fiber over $m\in \mathscr{M}$, we
	obtain a map $$\overline{\nabla}_m: H^0(X, \omega_X)\to H^1(X,
	\mathscr{O}_X)\otimes \Omega^1_{\mathscr{M}, m}.$$ As $\mathscr{M}\to
	\mathscr{M}_g$ is \'etale, there is a natural identification of
	$\Omega^1_{\mathscr{M}, m}$ with $\Omega^1_{\mathscr M_g,[Y]} \simeq
	H^0(Y, \omega_Y^{\otimes 2})$. Applying Serre duality, we may view
	$\overline{\nabla}_m$ as a map $$H^0(X, \omega_X)\to \on{Hom}(H^0(X,
	\omega_X), H^0(Y, \omega_Y^{\otimes 2})),$$ or equivalently, as a map 
	\begin{align}
		\label{equation:period-map}
	H^0(Y, f_*\omega_X)\to \on{Hom}(H^0(Y, f_*\omega_X), H^0(Y, \omega_Y^{\otimes 2})). 
	\end{align}
	Then \cite[Theorem 5.1.6 and Lemma
	A.1.8]{landesmanL:canonical-representations},
	applied where the local system $\mathbb V$ in
	\cite[Notation 5.1.1]{landesmanL:canonical-representations}
	is taken to be $h_*\mathbb{Q}$, shows that \eqref{equation:period-map}
	is induced by the map 
\begin{align*}
	\alpha: f_* \omega_X & \rightarrow \underline{\hom}(f_* \omega_X, \omega_Y^{\otimes
	2})\\
	\eta & \mapsto \left( q_\eta: \nu \mapsto \on{tr}_{X/Y}(\eta \otimes \nu) \right)
\end{align*}
by taking global sections. Note that $\alpha$ is injective, since, by 
\cite[Theorem 5.1.6]{landesmanL:canonical-representations},
it is obtained from the isomorphism
$\beta: f_* \mathscr O_X \to \underline{\hom}(f_* \mathscr O_X, \mathscr O_Y)$
(which corresponds to self-duality of the regular representation)
by tensoring $\beta$ with powers of $\omega_Y$ to obtain the map $\alpha$ as the
composition of injective maps
\begin{align*}
	f_* \omega_X \simeq f_* \mathscr O_X \otimes \omega_Y \to
	\underline{\hom}(f_* \mathscr O_X \otimes \omega_Y, \omega_Y^{\otimes
	2}) \simeq \underline{\hom}(f_* \omega_X, \omega_Y^{\otimes
	2}).
\end{align*}
	
	As $\overline{\nabla}$ has nontrivial kernel, the same is true for
	$\overline{\nabla}_m$. That is, there exists nonzero $\eta\in H^0(Y,
	f_*\omega_X)$ such that the nonzero map $q_\eta: f_*\omega_X\to
	\omega_Y^{\otimes 2}$ induces the zero map on global sections. Hence any global
	section of $f_*\omega_X$ lies in the kernel of $q_\eta$. Said another
	way, $f_*\omega_X$ is not generically globally generated.
\end{proof}

It remains to show that there are versal families of covers of curves of genus $2$ so that $\on{Pic}^0_{\mathscr X/\mathscr M}$ has an
isotrivial isogeny factor. We now do this carefully, but note it
can also be extracted from \cite[\S3]{BT:updated-curve-correspondences},
culminating in 
\cite[Example 3.7 and Proposition 3.8]{BT:updated-curve-correspondences}.
We include a proof for completeness, following \cite{BT:updated-curve-correspondences}.

\begin{proposition}
	\label{proposition:bt}
	There exists a versal family of covers of curves of genus $2$, $\mathscr
	X \xrightarrow{h} \mathscr Y \xrightarrow{\pi} \mathscr M,$ 
	with $\pi' = h \circ \pi$
	as in
	\autoref{notation:versal},
such that $h$ has degree $36$, $\on{Pic}^0_{\mathscr{X}/\mathscr{M}}$ has an isotrivial isogeny factor,
and the map $\mathscr M \to \mathscr M_2$ induced by $\pi$ is surjective and \'etale.
\end{proposition}
\begin{remark}
	\label{remark:}
	In fact, \autoref{proposition:bt} has a straightforward generalization
	to higher genus:
	Let $\mathscr H_g$ denote the moduli stack of
	hyperelliptic curves of genus $g$.
	There exists a family $\mathscr X \xrightarrow{h} \mathscr Y
	\xrightarrow{\pi}\mathscr M$ so that $\pi$ is a family of smooth proper
	genus $g$ hyperelliptic curves with geometrically connected fibers, $\pi \circ h$ is a family of smooth
	proper curves with geometrically connected fibers, $h$ is finite \'etale of degree $36$,
	$\on{Pic}^0_{\mathscr{X}/\mathscr{M}}$ has an isotrivial isogeny factor,
	and $\mathscr M \to \mathscr H_g$ is a surjective \'etale map.
\end{remark}

\begin{proof}[Proof of \autoref{proposition:bt}]
	\begin{figure}[h!]
	\begin{equation}
	\label{equation:}
	\begin{tikzcd} [row sep=14 pt]
		\mathscr{Y} \ar {d} & \mathscr{C}_1 \ar {d} \ar{l} &
		\mathscr{C}_2 \ar{l} \ar{d} & \mathscr X \ar{dd}
		\ar{l} \\
		\mathbb P  & \mathscr E\ar{l}{\phi} & \mathscr E \ar{l}{[\times 3]} \ar{d}{\alpha}
		&  \\
		& & \mathbb P ( q_*\mathscr O_{\mathscr E}(2 t_5)) & \mathscr
		E_0 \ar{l}
\end{tikzcd}\end{equation}
	\caption{A diagram depicting the relevant curves in the proof of
	\autoref{proposition:bt}.}
	\label{figure:bt}
\end{figure}

First, we construct a particular scheme $\mathscr M$, which has a surjective \'etale map to
$\mathscr M_2$. Let $\mathscr{M}'$ be the $S_6$-cover of $\mathscr{M}_2$
parametrizing orderings of Weierstrass points on the universal curve, and let
$\psi': \mathscr{Y}'\to \mathscr{M}'$ be the pullback of the universal curve to $\mathscr{M}'$. 
Let $\mathbb{P}':=\mathbb{P}(\psi'_*\omega_{\mathscr{Y}'/\mathscr{M}'})$, so that
there is a natural $2$-to-$1$ map $\mathscr{Y}'\to \mathbb{P}'$, ramified over the
images of $6$ disjoint sections $s_1', \cdots, s_6': \mathscr M' \to \mathscr Y'$. 
The statement of \autoref{proposition:bt}
is insensitive to replacing $\mathscr{M}'$ with a Zariski-open cover, and replacing $\mathscr
Y', \mathbb P', s_i'$ by their pullbacks to this cover, and we do so freely.
Zariski-locally on $\mathscr{M}'$, we may construct the double cover $p': \mathscr{E}'\to \mathbb{P}'$
branched over the images of $s_1', \ldots, s_4'$ in
$\mathbb P'$.
Now, let $\mathscr M :=(p')^{-1}(s_5'(\mathscr{M}'))$ denote the finite \'etale
double cover of $\mathscr M'$
where one additionally marks a point of $\mathscr{E}'$ mapping to the image of $s_5'$ under $p'$. 
Let $\mathbb P, \mathscr Y, \mathscr E$ denote the
pullbacks of $\mathbb P', \mathscr Y', \mathscr E'$ along $\mathscr M \to \mathscr
M'$, and let $q:\mathscr{E}\to \mathscr{M}$ be the natural map. By construction,
$q$ has a section, call it $t_5$, whose image lies over the image of $s_5'$. 
We consider $(\mathscr E, t_5)$ as an elliptic curve with identity section
$t_5: \mathscr M \to \mathscr E$.

The next several steps in the proof construct a sequence of three finite \'etale
covers of $\mathscr Y$, the last of which maps to an isotrivial elliptic
curve $\mathscr E_0$, as in \autoref{figure:bt}.
Let $\mathscr{C}_1$ be the normalization of the fiber product $\mathscr{Y}\times_\mathbb{P} \mathscr{E}$. 
We claim $\mathscr{C}_1$ is finite \'etale over $\mathscr{Y}$. 
To see this,
observe $\mathscr Y \to \mathbb P$ is branched to order $2$ at every
point in the branch locus of the map $\phi: \mathscr{E}\to \mathbb{P}$ obtained by pulling back $p'$.
Therefore, $\mathscr C_1 \to \mathscr Y$ is finite \'etale
by a relative version of Abhyankar's
lemma \cite[Expose\'e XIII, Proposition 5.5]{sga1}.

Next, define $\mathscr{C}_2:= \mathscr{C}_1
\times_{\mathscr E, [\times 3]} \mathscr E$, where the map $[\times 3]: \mathscr
E \to \mathscr E$ is multiplication by $3$ on the relative elliptic curve, and where
we use $t_5$ as the identity section of the elliptic curve $\mathscr E$.
Because $[\times 3]$ is finite \'etale,
$\mathscr{C}_2$ is finite \'etale over $\mathscr C_1$, hence over $\mathscr Y$.

We next construct one further finite \'etale cover $\mathscr X$ of $\mathscr C_2$.
Let $\alpha : \mathscr E \to \mathbb P (q_*\mathscr O_{\mathscr E}(2 t_5))$ denote the map induced by the
complete linear system associated $2t_5$. The map $\alpha$ is a double cover
ramified along the $2$-torsion of the relative elliptic curve $(\mathscr E,
t_5)$ over $\mathscr M$.
Let $\mathscr{D}=\alpha(\mathscr{E}[3]\setminus \on{im}(t_5))$, and note that
$\mathscr{D}$ is finite \'etale of degree $4$ over $\mathscr{M}$. 
Zariski-locally on $\mathscr{M}$, we can and do construct the 
double cover $\mathscr E_0 \to \mathbb P
( q_*\mathscr O_{\mathscr E}(2 t_5))$, branched over $\mathscr{D}$,
which is a family of 
genus $1$ curves. We replace $\mathscr{M}$ with the above Zariski cover, used to
construct $\mathscr E_0$.
Then, $\mathscr{X}$, defined as the normalization of $\mathscr{E}_0
\times_{\mathbb P (q_*\mathscr O_{\mathscr E}(2 t_5))} \mathscr C_2$, has a dominant map to
$\mathscr{E}_0$.
To conclude the proof, it is enough to show $\mathscr E_0$ is isotrivial and
$\mathscr{X} \to \mathscr{C}_2$ is finite \'etale of degree $36$.
Indeed, since $\mathscr X \to \mathscr E_0$ is a surjective map, $\on{Pic}^0_{\mathscr X/\mathscr M}$ has $\mathscr E_0$ as an
isogeny factor, which we will show to be isotrivial.

First, $h$ is a composite of $3$ maps of degrees $2,9$, and $2$, so $h$ has degree $36$.

Next, we claim $\mathscr{X}\to \mathscr{C}_2$ is finite \'etale.
Since $\mathscr C_1 \to \mathscr E$ is
branched to order $2$ over $t_5$, $\mathscr C_2 \to \mathscr E$ is
branched to order $2$ over $\mathscr E[3]$.
Hence,
$\mathscr{C}_2\to \mathbb P
( q_*\mathscr O_{\mathscr E}(2 t_5))$ is evenly branched over $\mathscr{D}$,
meaning that each preimage of any point in $\mathscr D$ has even ramification order.
Therefore, $\mathscr{X}\to \mathscr{C}_2$ is finite \'etale
by relative Abhyankar's
lemma \cite[Expose XIII, Proposition 5.5]{sga1}.

It remains to show $\mathscr{E}_0$ is isotrivial.
This follows from the computation preceding \cite[Lemma
4.3]{BT:updated-curve-correspondences}, which we now recall.
Consider the Hesse pencil $$E_\lambda: x^3+y^3+z^3+\lambda xyz=0,$$ 
where we view $E_\lambda$ as a family of elliptic curves with identity point
$[1:-1:0] \in E_\lambda \subset \mathbb P^2$.
Projecting $E_\lambda$
away from the point $[1:-1:0]$, we obtain a double cover $E_\lambda\to
\mathbb{P}^1$, given as the quotient $E_\lambda\to E_\lambda/\{\pm
1\}\simeq \mathbb{P}^1$.
Since $E_\lambda[3]$ is precisely the base locus of the Hesse pencil,
the image of $E_\lambda[3]$ is independent of $\lambda$.

We next claim that any elliptic curve over $\mathbb C$ is isomorphic to $E_\lambda$ for some
$\lambda \in \mathbb C$.
Note that the family $(E_\lambda)_{\lambda \in \mathbb P^1}$ defines a relative 
curve over $\mathbb P^1_\lambda$,
whose fiber at $\lambda = \infty$ is the reducible nodal curve $xyz = 0$.
The other singular members of the Hesse pencil are also nodal curves with three
irreducible components.
This family therefore corresponds to a map $\upsilon: \mathbb P^1 \to \overline{\mathscr
M}_{1,1}$, where $\overline{\mathscr M}_{1,1}$ is the moduli stack of elliptic
curves. The map is surjective because it induces a nonconstant map from $\mathbb
P^1$ to the coarse moduli space of $\overline{\mathscr M}_{1,1}$.
Therefore, every smooth elliptic curve over the complex numbers is isomorphic to
$E_\lambda$ for some $\lambda \in \mathbb C$.

Since any two elliptic curves $E$ and $E'$ appear as members of the Hesse pencil, the
images $E[3] \to E/\{\pm 1\} \simeq \mathbb P^1$ and $E'[3] \to E'/\{\pm 1\}
\simeq \mathbb P^1$ have
the same cross ratios.
Hence, any two fibers of 
the pair $(\mathbb P ( q_*\mathscr O_{\mathscr E}(2 t_5)), \mathscr{D})$ are
isomorphic,
so any two fibers of
$\mathscr E_0$
are isomorphic. Hence, $\mathscr E_0$ is isotrivial.
\end{proof}

We now straightforwardly combine the above results to
prove \autoref{theorem:prill-example}.
\begin{proof}[Proof of \autoref{theorem:prill-example}]
	By \autoref{proposition:bt}, there is a family $h: \mathscr X \to \mathscr
	Y$ of degree $36$ finite \'etale covers of genus $2$ curves over $\mathscr M$, where the induced
	map $\mathscr M \to \mathscr M_2$ is surjective \'etale and
	$\on{Pic}^0_{\mathscr X/\mathscr M}$ has an isotrivial isogeny factor.
	For any point $m \in \mathscr M$, let $f: X \to Y$ 
	be the fiber of $h :
	\mathscr X \to \mathscr Y$ over $m$.
	By \autoref{proposition:non-ggg},
	$f_* \omega_X$ is not generically globally
	generated. 
	By \autoref{proposition:ggg-and-prill}, $X \to Y$ is Prill exceptional.
\end{proof}

\bibliographystyle{alpha}
\bibliography{bibliography-mcg-hodge-theory}

\end{document}